\documentclass{amsart2010}
\usepackage{graphicx}
\usepackage{amsmath}
\usepackage{amsfonts}
\usepackage{amssymb}
 \usepackage{color}

\newtheorem{thm}{Theorem}
\newtheorem{cor}[thm]{Corollary}
\newtheorem{lem}[thm]{Lemma}
\newtheorem{pps}[thm]{Proposition}
\theoremstyle{definition}

\newtheorem*{theorem*}{Theorem}

\theoremstyle{remark}

\numberwithin{equation}{section}

\def\<{\langle}
\def\>{\rangle}
\begin{document}
	
\title[]{Complex symmetry of composition operators on weighted Bergman spaces  }

\author{Osmar R Severiano}%

\address{IMECC, Universidade Estadual de Campinas, Campinas-SP, Brazil.}

\email{$\mathrm{osmar.rrseveriano@gmail.br}$}

\begin{abstract} In this article, we study the complex symmetry of compositions operators $C_{\phi}f=f\circ \phi$ induced on weighted Bergman spaces $A^2_{\beta}(\mathbb{D}),\ \beta\geq -1,$ by analytic self-maps of the unit disk. One of ours main results shows that $\phi$ has a fixed point in $\mathbb{D}$ whenever $C_{\phi}$ is complex symmetric. Our works establishes a strong relation between complex symmetry and cyclicity. By assuming $\beta\in \mathbb{N}$ and $\phi$ is an elliptic automorphism of $\mathbb{D}$ which not a rotation, we show that $C_{\phi}$ is not complex symmetric whenever $\phi$ has order greater than $2(3+\beta).$

\end{abstract}

{\subjclass[2010]{Primary; Secondary}}
\maketitle{}

\section*{Introduction}
If $X$ is a Banach space of analytic functions on an open set $U\subset \mathbb{D}$ and if $\phi$ is an analytic self-map of $U,$ the \textit{composition operator} with \textit{symbol} $\phi$ is defined by $C_{\phi}f=f \circ \phi$ for any $f \in X.$ The emphasis here is on the comparison of properties of $C_{\phi}$ with those of symbols $\phi.$ Composition operators have been studied on a variety of spaces and the majority of the literature is concerned with sets $U$ that are open and connected. It clear that the set $U$ strongly influences the properties of the operator $C_{\phi}.$ For example if $U$ is the open unit disk $\mathbb{D},$ it is well-known that the operator $C_{\phi}$ is bounded on the Hardy space $H^2(\mathbb{D}).$ In general this result holds for each weighted Bergman space $A^2_{\beta}(\mathbb{D})$ (see \cite[page 532]{anna}). We refer to \cite{cow} and \cite{hed} for more details about the Hardy and weighted Bergman space respectively.

 The concept of \textit{complex symmetric operators} on separable Hilbert spaces is a natural generalization of
complex symmetric matrices, and their general study was initiated by Garcia, Putinar, and Wogen (see \cite{Garc2},\cite{Garc3}, \cite{Wogen1} and \cite{Wogen2}). The class of complex symmetric operators includes
a large number of concrete examples including all normal operators, binormal operators, Hankel operators, finite Toeplitz matrices, compressed shift operators, and
the Volterra integral operator.

The study of complex symmetry of composition operators on the Hardy
space of the unit disc $H^2 (\mathbb{D})$ was initiated by Garcia and Hammond in \cite{Garc1}. In this work, they  showed that for each $\alpha\in \mathbb{D},$ the involutive automorphism of $\mathbb{D}$ given by 
\begin{align}\label{0}
\phi_{\alpha}(z):=\frac{\alpha-z}{1-\overline{\alpha}z}
\end{align}
induces a non-normal complex symmetric composition operator. In particular, we see that the class of
complex symmetric operators is strictly larger than that of the normal operators. Another important work on complex symmetry of composition operators on $H^2(\mathbb{D})$ was realized by P. S. Bourdon and S. W. Noor in \cite{wal}. In this work they showed that if $\phi$ is an elliptic automorphism of order $N>3$ (including $N=\infty$) then $C_{\phi}$ is not complex symmetric \cite[Proposition 3.1.]{wal} and \cite[Proposition 3.3]{wal}. It is worth mentioning that for a complete classification of the automorphisms of $\mathbb{D}$ that induce complex symmetric composition operators it is sufficient to classify the elliptic automorphism of order 3. Based on \cite{wal}, T. Eklund, M. Lindström and P. Mleczko also tried to classify which automorphisms of $\mathbb{D}$ induce complex symmetric composition operators on the classical Bergman space $A^2.$ They showed that if $\phi$ is an elliptic automorphism of order $N>5$ then $C_{\phi}$ is not complex symmetric.  

Our first main result is the following:
\begin{thm}\label{pps1}
Let $\phi$ be an analytic self-map of $\mathbb{D}.$ If $C_{\phi}$ is complex symmetric on $A^2_{\beta}(\mathbb{D}),$ then $\phi$ must fix a point in $\mathbb{D}.$ 
\end{thm}

  We will use Theorem \ref{pps1} to prove that the complex symmetry of $C_{\phi}$ strongly influences the dynamics of $C_{\phi}$ and $C_{\phi}^*$ on $A^2_{\beta}(\mathbb{D})$ (see Propositions \ref{pps4} and \ref{pps3}). As a consequence we will show that \textit{hyperbolic linear fractional maps} of $\mathbb{D}$ never induce complex symmetric composition operators. Hence when $\phi$ is a \textit{parabolic} or \textit{hyperbolic} automorphism of $\mathbb{D}$ we will see that $C_{\phi}$ is not complex symmetric. Our main result on the complex symmetry of composition operators induced by elliptic autmorphisms generalizes the results in \cite{wal} and \cite{Ted} on $A^2_{-1}:=H^2(\mathbb{D})$ and $A^2:=A^2_{0}(\mathbb{D})$ respectively, to all $A^2_{\beta}(\mathbb{D})$ with $\beta\in \mathbb{N}.$ We prove the following result:

\begin{thm}
	Let $\beta\in \mathbb{N}.$ If $\phi$ is an elliptic automorphism of $\mathbb{D}$ which is not a rotation then $C_{\phi}$ is not complex symmetric on $A^2_{\beta}(\mathbb{D})$ if $\phi$ has order $N\geq2(3+\beta).$
\end{thm}
\section{Notations and Preliminaries }
In this section, we present some preliminary definitions and results. Throughout this article we will use the following notations:   $\mathbb{D}:=\{z\in\mathbb{C}:\left| z\right| <1\}$ the open unit disc,   $\mathbb{T}:=\{z\in \mathbb{C}:\left| z\right| =1\}$ unit circle, $\mathbb{N}:=\{0,1,2,\ldots\}$ and for each operator $T$ on Hilbert space we denote the orbit of $T$ in $f$ by $\mathrm{Orb}(T,f)=\{T^nf:n=0,1,\ldots\}.$
\subsection{Complex symmetric operators} A bounded operator $T$ on a separable Hilbert space $\mathcal{H}$ is \textit{complex symmetric} if
there exists an orthonormal basis for $\mathcal{H}$ with respect to which $T$ has a self-transpose
matrix representation. An equivalent definition also exists. An conjugate-linear operator $C$ on $\mathcal{H}$ is said to be a \textit{conjugation} if $C^2=I$ and $\left\langle Cf,Cg\right\rangle =\left\langle g,f\right\rangle $ for all $f, g\in \mathcal{H}.$ So, we say that $T$ is $C$-\textit{symmetric} if $CT = T^*C,$ and complex symmetric if there
exists a conjugation $C$ with respect to which $T$ is $C$-symmetric.

In general, complex symmetric operators enjoy the following \textit{spectral symmetry } property :
 \begin{align}\label{37}
f\in Ker(T-\lambda I)\Longleftrightarrow Cf\in Ker(T^*-\overline{\lambda}I).
\end{align}
This follows from $CT=T^*C$ where $C$ is a conjugation. Another fact well-known in the literature says that $f$ is orthogonal to $Cg$  whenever $f$ and $g$ are eigenvectors corresponding to distinct eigenvalues of a $C$-symmetric operator (see \cite{pamona}).

\subsection{Weighted Bergman space $A^2_{\beta}(\mathbb{D})$} Let $dA(z)$ be the normalized area measure on $\mathbb{D}$ and $-1<\beta<\infty.$ The weighted Bergman space $A^2_{\beta}:=A^2_{\beta}(\mathbb{D})$ is the space of all analytic functions in $\mathbb{D}$ such that norm 
\begin{align*}
\left\|f\right\|:= \left( \int_{\mathbb{D}}\left| f(z)\right| ^2dA(z)\right) ^{1/2}<\infty,
\end{align*}
where $dA_{\beta}(z)=(\beta+1)(1-\left|z\right|^2)^{\beta}dA(z).$ The weighted Bergman space $A^2_{\beta}$ is a Hilbert space with the inner product 
\begin{align}\label{eq1}
\displaystyle\left\langle f,g\right\rangle=\sum_{n=0}^{\infty}\frac{n!\Gamma(2+\beta)}{\Gamma(n+2+\beta)}   \widehat{f}(n) \overline{\widehat{g}(n)},
\end{align}
where $(\widehat{f}(n))_{n\in \mathbb{N}}$ and $(\widehat{g}(n))_{n\in \mathbb{N}}$ are the sequences of Maclaurin coefficients for $f$ and $g$ respectively, and $\Gamma$ is the Gamma function.
Hence, the norm of $f\in A^2_{\beta}$ is also given by 
\begin{align}\label{eq2} 
\left\|f\right\|=\sqrt{\displaystyle \sum_{n=0}^{\infty}\frac{n!\Gamma(2+\beta)}{\Gamma(n+2+\beta)}\left|\widehat{f}(n)\right|^2}.
\end{align}
For convenience we write $A^2_0:=A^2,$ and we interpreted the classical Hardy space $H^2(\mathbb{D})$ as the \textit{limit case} of the weighted Bergman space $A^2_{\beta},$ $\beta\longrightarrow -1,$ that is $H^2(\mathbb{D}):=A^2_{-1}$ (see \cite{rikka}). 
For each $\alpha\in \mathbb{D},$ let $K_{\alpha}$ denotes the \textit{reproducing kernel} for $A^2_{\beta}$ at $\alpha;$ that is
\begin{align*}
K_{\alpha}(z)=\frac{1}{(1-\overline{\alpha}z)^{2+\beta}}.
\end{align*} 
These funcions play an important role in the theory of weighted Bergman spaces, namely: $\left\langle f, K_{\alpha}\right\rangle =f(\alpha)$ for all $f\in A^2_{\beta}.$ Moreover, if $\phi$ is an analytic self-map of $\mathbb{D},$ then a simple computation gives $C_{\phi}^*K_{\alpha}=K_{\phi(\alpha)}.$ For more details about $A^2_{\beta}$ we refer \cite{hed} and \cite{peter}.

 The space $H^{\infty}:=H^{\infty}(\mathbb{D})$ is the Banach space of all analytic and bounded functions on $\mathbb{D}.$ The norm of a function $f\in H^{\infty}$ is defined by $\left\| f\right\|_{\infty} =\sup\left\lbrace \left| f(z)\right| :z\in \mathbb{D}\right\rbrace.$ It is straightforward to verify that $H^{\infty}$ is a subspace of $A^2_{\beta}$ and $
\left\|f\right\|\leq \left\| f\right\| _{\infty}.$ Moreover, for each $\psi \in H^{\infty},$ we define the bounded operator  $M_{\psi}:A^2_{\beta}\longrightarrow A^2_{\beta}$ by $(M_{\psi}f)(z)=\psi(z)f(z).$
This operator is called (\textit{analytic}) $\mathit{Toeplitz}$  $\mathit{operator},$ it also is called multiplication operator by $\psi.$
\subsection{Denjoy-Wolff point}Let $\phi^{[n]}$ denote the $n$-th iterate of the analytic self-map $\phi,$ and we define inductively by $\phi^{[0]}=id,\phi^{[1]}=\phi$ and $\phi^{[n]}=\phi^{[n-1]}\circ \phi$ for each non-negative integer $n.$ Since $C_{\phi}^n=C_{\phi^{[n]}},$ it follows that the dynamic of $\phi$ influences strongly the dynamics of $C_{\phi}.$ If $\omega \in \overline{\mathbb{D}}$ is a fixed point for $\phi$ such that the sequence $\phi^{[n]}$ converges uniformly on compact subsets of $\mathbb{D}$ to $\omega,$ then $\omega$ is said to be an \textit{attractive fixed
point} for $\phi.$ The next result is concerned with the existence of attractive fixed points for analytic self-maps of $\mathbb{D}$ which are not elliptic automorphisms. 

\begin{thm}
	If $\phi$ is an analytic self-map of $\mathbb{D}$ is not an elliptic automorphism, then there is an unique point $\omega\in \overline{\mathbb{D}}$ such that 
	\begin{align}\label{t7}
	\omega=\lim_{n\longrightarrow \infty}\phi^{[n]}(z)
	\end{align}
	for each $z\in \mathbb{D}.$
\end{thm}
This result is proved in \cite{wolf1} and \cite{wolf2}. The point $\omega$ in \eqref{t7} is called the \textit{Denjoy-Wolff} point of $\phi.$ If $\omega\in \mathbb{D}$ then $\omega$ is the unique fixed point of $\phi$ in $\mathbb{D}.$

\subsection{Linear fractional composition operators}Recall that a \textit{linear fractional self-map} of $\mathbb{D}$ is a mapping of the form 
\begin{align}\label{frac1}
\phi(z)=\frac{az+b}{cz+d}
\end{align}
satisfying $	\left| b\overline{d}-a\overline{c}\right| +\left| ad-bc\right| \leq \left| d\right| ^2-\left| c\right| ^2
$ (see \cite{ele}). Let $\mathrm{LFT}(\mathbb{D})$ be denotes the set of all linear fractional self-maps of $\mathbb{D}.$ Since these maps have at one and most two fixed points in $\overline{\mathbb{D}},$ we classify them according to the location of their fixed points, namely:
\begin{itemize}
\item \textit{Parabolic maps:} If $\phi$ has their fixed point on $\mathbb{T}.$
\item \textit{Hyperbolic maps:} If $\phi$ has attractive fixed point on $\overline{\mathbb{D}}$ and the other fixed point outside of $\mathbb{D},$ and both fixed points are on $\mathbb{T}$ if and only if $\phi$ is an automorphism of $\mathbb{D}.$

\item \textit{Loxodromic and elliptic maps:} If $\phi$ has a fixed point on $\mathbb{D}$ and the other fixed point outside of $\overline{\mathbb{D}}.$ The elliptic ones are always automorphisms of $\mathbb{D}.$ 

	\end{itemize}
It is worth mentioning that the automorphisms of $\mathbb{D}$ are linear fractional maps of the form 
\begin{align}
\phi(z)=e^{i\theta} \frac{\alpha-z}{1-\overline{\alpha}z}
\end{align}
for some $\theta\in \mathbb{R}$ and $\alpha\in \mathbb{D}.$ If $\phi$ is an elliptic automorphism of $\mathbb{D},$ then $\phi$ has an unique fixed point $\alpha\in\mathbb{D}.$ In this case $\phi$ is conjugate to a rotation via Schwarz Lemma, that is
\begin{align}\label{pre2}
\phi=\phi_{\alpha}\circ(\lambda\phi_{\alpha})
\end{align}
for some $\lambda$ on the unit circle $\mathbb{T}.$ Following, we highlight some properties of involutive automorphism $\phi_{\alpha}$ once it will play an important role in Section \ref{5}. Since $C_{\phi_{\alpha}}$ is an invertible operator with $C_{\phi_{\alpha}}^2=I$ and $(z^n)_{n\in \mathbb{N}}$ has dense span on $A^2_{\beta},$ it follows that the sequence formed by the vectors $\phi_{\alpha}^n=C_{\phi_{\alpha}}z^n$ is dense on $A^2_{\beta}$ too, moreover putting $v_n:=C_{\phi_{\alpha}}^*z^n$ for each non-negative integer we have $\left\langle v_n,v_m\right\rangle =0$ if $n\neq m.$

If $\phi$ is an analytic self-map of $\mathbb{D}$, then no general formula for $C_{\phi}^{*}$ is known. A first result in this direction is due to Carl Cowen (see \cite[Theorem 9.2]{cow}). He showed that if $\phi$ is a linear fractional self-map of $\mathbb{D}$ then such a formula is given in terms of compositions of Toeplitz operators and compositons operators.  
Based on the work of Carl Cowen, Hurst in \cite[Theorem 2]{Hurst} generalized the formula for weighted Bergman spaces $A^2_{\beta};$  \textit{if $\phi$ is as in \eqref{frac1} then }
	\begin{align}\label{frac2}
C_{\phi}=M_{g}C_{\sigma}M_{h}^{*},
\end{align}
\textit{where the functions $g,h$ and $\sigma$ are defined as}
\begin{align*}
 \sigma(z)=\frac{\overline{a}z-\overline{c}}{-\overline{b}z+\overline{d}}, \   	g(z)=\frac{1}{(-\overline{b}z+\overline{d})^{\beta+2}}\   and \ h(z)=(cz+d)^{\beta+2}.
\end{align*}

Hence, if $\phi=\phi_{\alpha}$ the formula \eqref{frac2} gives  $C_{\phi_{\alpha}}^*=M_{K_{\alpha}}C_{\phi_{\alpha}}M_{1/K_{\alpha}}^*.$  So, assuming $2+\beta$ is a natural number, we determine an expression for $M_{1/K_{\alpha}}$ in terms of the multiplication operator $M_z$, by observing 
\begin{align}\label{sym2}
\frac{1}{K_{\alpha}(z)}=(1-\overline{\alpha}z)^{2+\beta}=\sum_{k=0}^{2+\beta} {2+\beta\choose k}(-\overline{\alpha}z)^{k}.
\end{align}
Then
\begin{align}\label{pre3}
M_{1/K_{\alpha}}=\displaystyle \sum_{k=0}^{2+\beta}{2+\beta\choose k}(-\overline{\alpha}M_z)^k.
\end{align}
By combining the formula for the adjoint of $C_{\phi_{\alpha}}$ and \eqref{pre3}, we obtain the following expression:
\begin{align}\label{658}
C_{\phi_{\alpha}}^*= \sum_{k=0}^{2+\beta}{2+\beta\choose k} (-\alpha)^k M_{K_{\alpha}}C_{\phi_{\alpha}}(M_z^*)^k.
\end{align}
The equality \eqref{658}  plays an important role in section \ref{5}, it will provide a 
general way to study the orthogonality between the vectors $v_n:=C_{\phi_{\alpha}}^*z^n$ 
as compared to \cite[Lemma 2.2.]{wal} and \cite[Lemma 5]{Ted}, where the authors considered the some problem on the Hardy space $H^2(\mathbb{D})$ and Bergman space $A^2$ respectively.

\section{Cyclic vectors on $A_{\beta}^2$}\label{2}

In this section, we focus on the study of \textit{cyclic vectors} for the weighted Bergman spaces $A^2_{\beta}$, we first establish some basic results. If $f\in A^2_{\beta},$ we let $[f]$ denote the closure in $A^2_{\beta}$ of $\{pf:\text{$p$ is a polynomial in $z$}\}.$ A function $f\in A^2_{\beta}$ is said to be \textit{cyclic} if it generates the whole space, that is, $[f]=A^2_{\beta}.$
  For the Hardy space $H^2(\mathbb{D})$ or the Bergman space $A^2,$ various sufficient  and necessary conditions are known to decide if a given function is cyclic. For example see \cite[Chapter 7]{hed} and \cite[Chapter 2]{mar}.

\begin{lem}\label{lem1} Let $M_z$ be the multiplication operator on $A_{\beta}^2.$ Then 
	\begin{align}\label{lem13}\left( M_z^*f\right) (z)=\sum_{n=0}^{\infty}\displaystyle\frac{\Gamma(n+2+\beta)(n+1)}{\Gamma(n+3+\beta)} \widehat{f}(n+1)  z^n
	\end{align}
for each $f\in A^2_{\beta}.$
\end{lem}

\begin{proof} It is enough to determine the coefficients of the function $M_{z}^*f$ That is, the value $\widehat{M_{z}^*f}(n)$ for each non-negative integer $n.$ We note that	\begin{align*}
	\frac{n!\Gamma(2+\beta)}{\Gamma(n+2+\beta)}\widehat{M^*_zf}(n) =
	\left\langle M^*_{z}f,z^n\right\rangle
	=\left\langle f,M_zz^{n}\right\rangle
	=\frac{(n+1)!\Gamma(2+\beta)}{\Gamma(n+3+\beta)}\widehat{f}(n+1).
	\end{align*}
	Hence, $\widehat{M^*_zf}(n)=\displaystyle\frac{\Gamma(n+2+\beta)(n+1)}{\Gamma(n+3+\beta)}\widehat{f}(n+1).$ 
\end{proof}

The relation \eqref{lem13} provides a simple formula to compute the $n$-th coefficient of Maclaurin of $\widehat{M^*_z}f,$ for each $f\in A^2_{\beta}.$ Below we use it to establish the injectivity of the operator $M_{\omega-z}^*.$ More precisely:

\begin{pps}\label{pps2} The adjoint of the multiplication operator $M_{\omega-z}$ on $A^2_{\beta}$ is injective for each $\omega\in \mathbb{T}.$ In particular, $(\omega-z)f$ is cyclic whenever $f$ is cyclic.
	
\end{pps}
\begin{proof} Suppose $M_{\omega-z}^*g=0$ for some $g\in A^2_{\beta}.$ Then by Lemma \ref{lem1} we have

\begin{align*}
\displaystyle \sum_{n=0}^{\infty}\overline{\omega}\widehat{g}(n)z^n=\sum_{n=0}^{\infty}\displaystyle\frac{\Gamma(n+2+\beta)(n+1)}{\Gamma(n+3+\beta)}\widehat{g}(n+1)z^n.
\end{align*}

Hence for each non-negative integer $n,$ we have 

\begin{align}\label{lem11}\widehat{g}(n+1)=\displaystyle \frac{\Gamma(n+3+\beta)}{\Gamma(n+2+\beta)(n+1)}\widehat{g}(n).
\end{align}

By using \eqref{lem11} recursively, we obtain the relation, 
$
\widehat{g}(n)=\displaystyle \frac{\Gamma(n+2+\beta)}{n!\Gamma(\beta+2)}\widehat{g}(0).
$ Thus a simple computation provides 
\begin{align*}
\left\|g\right\|^2=  \sum_{n=0}^{\infty}\frac{n!\Gamma(2+\beta)}{\Gamma(n+2+\beta)}\left|\widehat{g}(n)\right|^2
=  \left[ \sum_{n=0}^{\infty} \frac{\Gamma(n+2+\beta)}{\Gamma(\beta+2)n!}\right] \left|\widehat{g}(0)\right|^2.
\end{align*}
Since $\displaystyle\sum_{n=0}^{\infty}\frac{\Gamma(n+2+\beta)}{\Gamma(\beta+2)n!}$ diverges, $\left\| g\right\| $ is finite if and only if $\widehat{g}(0)=0.$ Thus all Maclaurin coefficients of $g$ must vanish. Hence $g\equiv 0,$  and therefore $M_{\omega-z}^*$ is injective.

Now suppose that $f$ is cyclic and $g\in [(\omega-z)f]^{\perp}.$ Then $ 0=\left\langle z^n(\omega -z)f,g\right\rangle$ for each non-negative $n.$ Hence
\begin{align}\label{lem12}0=\left\langle z^n(\omega -z)f,g\right\rangle =\left\langle M_{\omega -z}(z^nf),g\right\rangle =\left\langle z^nf,M_{\omega -z}^*g\right\rangle. 
\end{align}
Since $f$ is cyclic, \eqref{lem12} forces $M_{\omega-z}^*g=0.$ By assuming $M_{\omega-z}^*$ injective, we get $g\equiv 0$. In particular, we conclude that $(\omega-z)f$ is cyclic.

\end{proof}

Theorem \ref{thm1} is the main result of this section. It generalizes similars results for composition operators on $H^2(\mathbb{D})$ and $A^2$ proved in \cite[Proposition 2.1]{bourdon} and \cite[Lemma 1]{Ted} respectively. The main tool they use is the cyclicity of $(\omega -z)g$ where $g$ is a cyclic eigenvector for $C_{\phi}$ and $\omega\in \mathbb{T}$ is the Denjoy-Wolff point of $\phi.$ 

\begin{thm}\label{thm1} Suppose that the analytic self-map $\phi$ of $\mathbb{D}$ has a Denjoy-Wolff point $\omega $ in $\mathbb{T}.$ If $\lambda$ is an eigenvalue of $C_{\phi}:A^2_{\beta}\longrightarrow A^2_{\beta}$ with a cyclic function as a corresponding eigenvector, then $C_{\phi}-\lambda I$ has dense range. 
\end{thm}
\begin{proof} Since $\phi$ has Denjoy- Wolff point in $\mathbb{T},$ it follows that $\phi$ is nonconstant, and therefore $\phi$ is an open function. Let $g$ be a cyclic eigenvector of $C_{\phi}$ corresponding to the eigenvalue $\lambda.$
It is worth noting that $\lambda\neq 0.$ Indeed, if $\lambda=0$ then $g(\phi(z))=0$ for all $z\in \mathbb{D}$ hence $g\equiv 0.$

 We recall that the operator $C_{\phi}-\lambda I$ has dense range if and only if $C_{\phi}^{*}-\overline{\lambda}I$  is injective. 
Therefore, if $\overline{\lambda}$ is an eigenvalue of $C_{\phi}^{*}$ there is a nonzero vector $h\in A^2_{\beta}$ such that $C_{\phi}^{*}h=\overline{\lambda}h$. Then for any non-negative integers $n$ and $k,$ we have
\begin{align}\label{the31}
\lambda^k\left\langle z^n(\omega -z)g,h\right\rangle=&\left\langle z^n(\omega -z)g, \overline{\lambda}^kh\right\rangle\nonumber
=\left\langle z^n(\omega -z)g,(C_{\phi}^{*})^{k}h\right\rangle\nonumber\\
= &\left\langle C_{\phi^k}(z^n(\omega -z)g),h\right\rangle\nonumber
=\left\langle \phi_{k}^n(\omega -\varphi_k)g\circ \phi_{k}^n,h\right\rangle\nonumber\\
=& \left\langle \phi_{k}^n(\omega -\phi_k)(C_{\phi})^kg,h\right\rangle\nonumber
=\left\langle \phi^k(\omega -\phi_k)\lambda^kg,h\right\rangle\nonumber\\
=&\lambda^k\left\langle \phi_k^{n}(\omega -\phi_k)g,h\right\rangle.
\end{align}
By combining \eqref{the31} and $\lambda\neq 0$ we obtain 
\begin{equation}\label{eq12}
\left\langle z^n(\omega -z)g,h\right\rangle=\left\langle \phi^n_{k}(\omega -\phi_k)g,h\right\rangle.
\end{equation}
Since $\left|\phi_{k}^n(\omega -\phi_k)g\overline{h}\right|\leq 2\left|g\overline{h}\right|\in L^1(\mathbb{D},dA_{\beta})$ and the iterated sequence $(\phi_k)_{k=1}^{\infty}$ converges pointwise to $\omega $ on $\mathbb{D}$ (even uniformly on compact subsets of $\mathbb{D},$ see \cite[Theorem 2.51]{cow}). By applying the Lebesgue Dominated Convergence Theorem, we obtain from \eqref{eq12} that
\begin{equation*}
\left\langle z^n(\omega -z)g,h\right\rangle=\displaystyle\lim _{k\longrightarrow \infty}\left\langle \phi_k^n(\omega -\phi_k)g,h\right\rangle=0,
\end{equation*}
for each non-negative integer $n.$ Hence, $h\in [(\omega-z)g]^{\perp}.$ By the Proposition \ref{pps2}, $(\omega-z)g$ is cyclic, and therefore $h$ is identically zero. However this contradicts the fact that $h$ is an eigenvector.
\end{proof}
\begin{cor}\label{cor1}Suppose that $\phi$ is an analytic self-map of $\mathbb{D}.$ If $\phi$ has Denjoy-Wolff point in $\mathbb{T},$ then $C_{\phi}-I$ has dense range on $A^2_{\beta}.$
\end{cor}
\begin{proof} This is an immediate consequence of Theorem \ref{thm1}, since
$\lambda=1$ is an eigenvalue for $C_{\phi}$ having the cyclic eigenvector $g\equiv 1.$
\end{proof}
\section{Cylicity and Hypercyclicity}\label{3}
The next main result shows that if $C_{\phi}$ is complex symmetric on $A^2_{\beta},$ then $\phi$ must fix a point in $\mathbb{D}.$ Naturally, this implies results about the dynamics of $C_{\phi}$ and $C_{\phi}^*.$ 

\begin{thm}\label{thm2} Let $\phi$ be a analytic self-map of $\mathbb{D}.$ If $C_{\phi}:A^2_{\beta}\longrightarrow A^{2}_{\beta}$ is complex symmetric then $\phi$ either an elliptic automorphism of the unit disc or has a Denjoy-Wolff point in $\mathbb{D}.$ 
\end{thm}
\begin{proof} Suppose on the contrary that $\phi$ has Denjoy-Wolff point in $\mathbb{T}.$ By Corollary \ref{cor1} follows that $C_{\phi}-I$ has dense range on $A^2_{\beta}$ or equivalently $C_{\phi}^*-I$ is injective. If $C_{\phi}$ also is complex symmetric, we obtain $CC_{\phi}^*C=C_{\phi}$ for some conjugation $C,$ so $C_{\phi}^*-I$ is not injective because $C_{\phi}^*CK_0=CK_{0}.$ Hence $C_{\phi}$ is not complex symmetric.
	\end{proof}

An operator $T$ on  $\mathcal{H}$ is said to be \textit{cyclic} if there exists a vector $f\in \mathcal{H}$ for which the linear span of its orbit $(T^nf)_{n\in \mathbb{N}}$ is dense in $\mathcal{H}.$ If the orbit $(T^nf)_{n\in \mathbb{N}}$ itself is dense in $\mathcal{H},$ then $T$ is said to be \textit{hypercyclic}. In these cases $f$ is called a \textit{cyclic} or \textit{hypercyclic} vector for $T$ respectively. The book \cite{caos} has a systematic study of cyclic and hypercyclic operators. In particular it shows how spectral properties influence dynamical properties. For example if $T$ has eigenvalues then $T^*$ is never hypercyclic (see \cite[Proposition 5.1.]{caos}). Hence $C_{\phi}^*$ is never hypercyclic since $1$ is always an eigenvalue for $C_{\phi}.$ 
Additionally, if $C_{\phi}$ is complex symmetric then we have:
\begin{pps}\label{pps4}
	Let $\phi$ be an analytic self-map of $\mathbb{D}$ such that $C_{\phi}$ is complex symmetric on $A^2_{\beta}.$ Then $C_{\phi}$ and $C_{\phi}^*$ are not hypercyclic.  
\end{pps}
\begin{proof} By the comments above it is enough to show that $C_{\phi}$ is not hypercyclic.  If $C_{\phi}$ is complex symmetric then $\phi(\alpha)=\alpha$ for some $\alpha \in \mathbb{D}.$ So $C_{\phi}^*K_{\alpha}=K_{\alpha},$ and therefore $C_{\phi}$ is not hypercyclic.
\end{proof}

Proposition \ref{pps4} says that the Bergman space $A^2_{\beta}$ does not support composition operators simultaneously complex symmetric and hypercyclic. As we saw this is strongly influenced by the existence of fixed points for $\phi$ in $\mathbb{D}.$

In contrast to Proposition \ref{pps4}, we will see that each complex symmetric  $C_{\phi}$ with non-automorphic $\phi$ is cyclic. This is consequence of the following result: \textit{If $\phi$ is an analytic self-map of the disk  with $\phi(\alpha)=\alpha$ for some
	a in the unit disk and $\phi$ is neither constant nor an elliptic automorphism, then $C_{\phi}^*$ is cyclic on $H^2(\mathbb{D})$ with cyclic vector $K_z$ for each $z\neq \alpha$}. The proof of this result appears in the work done by T. Worner (see \cite{Worner}), where the author studied the commutant of certains composition operators in $H^2(\mathbb{D}).$ Here it is worth mentioning that the proof given in \cite[Theorem 3]{Worner} works for $A^2_{\beta}.$

\begin{pps}\label{pps3}
	Let $\phi$ be an analytic self-map of $\mathbb{D}$ such that $C_{\phi}$ is complex symmetric on $A^2_{\beta}$ and $\phi$ is neither constant nor an elliptic automorphism then $C_{\phi}$ and $C_{\phi}^*$ are cyclic. 
\end{pps}
\begin{proof} Let $\alpha$ be the Denjoy-Wolff of $\phi.$ According to Theorem \ref{thm2}, the complex symmetry of $C_{\phi}$ implies $\alpha\in \mathbb{D}.$ In particular, $\phi(\alpha)=\alpha.$ So $C_{\phi}^*$ is cyclic and $K_z$ is a cyclic vector for $C_{\phi}^*$ for each $z\neq\alpha.$ Since $C_{\phi}$ is complex symmetric, we have $C_{\phi}=CC_{\phi}^*C$ for some conjugation $C.$ Let $f$ in the orthogonal complement of the span of $\mathrm{Orb}(C_{\phi},CK_{z})$ then
	\begin{equation*}0=\left\langle f,C_{\phi}^{n}CK_{z}\right\rangle=\left\langle f,C(C _{\phi}^n)^*K_z\right\rangle=\left\langle (C_{\phi}^{n})^*K_z,Cf\right\rangle
	\end{equation*}
	for each non-negative integer $n.$ So $Cf=0$ since $K_z$ is a cyclic vector for $C_{\phi}^*.$ Because $C$ is an isometry, we obtain $f\equiv 0.$  Hence, span of $\mathrm{Orb}(C_{\phi},CK_{z})$ is dense in $A^2_{\beta}$ and therefore $C_{\phi}$ is cyclic.
\end{proof}	
It is worth mentioning that Propositions \ref{pps4} and \ref{pps3} are generalizations of \cite[Theorem 5.1.]{jung}. Moreover, we need not suppose that $\phi$ has a fixed point in $\mathbb{D},$ because this is by guaranteed Theorem \ref{thm2}.

\section{Hyperbolic linear fractional non-automorphisms}\label{4}
As we saw in Section \ref{3} the complex symmetry of $C_{\phi}$ strongly influences the location of the Denjoy-Wolff point of $\phi.$ More precisely, if $C_{\phi}$ is complex symmetric then $\phi$ has a fixed point inside the disk (see Proposition \ref{thm2}). Hence parabolic linear fractional maps never induce complex symmetric composition operators on $A^2_{\beta}.$ So, for a complete classification of the linear fractional self-maps that induce complex symmetric composition operators on $A^2_{\beta}$ we must study the hyperbolic, loxodromic and elliptic maps.

 In this section we deal with the case in which $\phi$ is a hyberbolic linear fractional map. We will see that in this case $C_{\phi}$ is not complex symmetric. We begin our study showing that each hyperbolic linear fractional map is similar to 
a map of the form  
\begin{align*}
\psi_s(z)=\frac{sz}{1-(1-s)z} \ \text{for some} \ 0<\left|s\right|<1.
\end{align*}

\begin{lem}\label{11}
	Let $\phi$ a hyperbolic linear fractional map. Then it is similar to $\psi_s$ for some $0<\left|s\right|<1.$
\end{lem}

\begin{proof}
	Let $\phi(\alpha)=\alpha$ for some $\alpha\in \mathbb{D}$ and $\phi(e^{i\theta})=e^{i\theta}.$ Putting $\phi_{\alpha}(e^{i\theta})=\lambda,$ we have $\left| \lambda\right|=1.$ Hence, the mapping $\Phi=\overline{\lambda}\phi_{\alpha}$ is an automorphism of $\mathbb{D}$ with
	$\Phi(\alpha)=0$ and $\Phi(e^{i\theta})=1.$ Since $\mathrm{LFT}(\mathbb{D})$ is a group with the composition of functions, it follows that $\psi=\Phi\circ \phi\circ \Phi^{-1}$ is a linear fractional map of $\mathbb{D}.$ Also $\psi(1)=1$ and $\psi(0)=0.$  Let $a,b,c,d$ be complex numbers such that $\psi(z)=(az+b)(cz+d)^{-1}.$ Then
	we obtain $b=d\psi(0)=0$ and $a/(c+d)=\psi(1)=1.$ Putting $s=(c+d)/d,$ we can rewrite $\psi$ as follows
	\begin{center}
		$
		\displaystyle \psi(z)=\frac{(c+d)z}{cz+d}=\frac{ \left( \frac{c+d}{d}\right)z}{1-\left( -\frac{c}{d}\right)z }=\frac{sz}{1-(1-s)z}=\psi_s(z).$
	\end{center}
	Now it is enough to prove that $0<\left|s\right|<1.$ If $s=0,$ the function $\psi_s$ is identically zero however this contradicts $\psi_s(1)=1.$ If $s=1,$  $\psi_s$ is the identity, and therefore $\phi$ is the identity too. As $\psi_s$ is a linear fractional map of $\mathbb{D},$ we have
	 \begin{align}\label{36}
	\left|s\right|\left(1+\left|1-s\right|\right)\leq \left(1-\left|1-s\right|\right)\left(1+\left|1-s\right|\right).
	\end{align}
	Since $s\neq 1,$ the inequality \eqref{36} forces $\left|s\right|\leq 1-\left|1-s\right|<1.$ This concludes the result.
\end{proof}

Since each hyperbolic linear fractional non-automorphism is similar to $\psi_s,$ we focus our study on the map $\psi_s.$ It is clear that $\psi_s$ has its Denjoy-Wolff point inside the unit disk, so by the discussion before Proposition \ref{pps3} the operator $C_{\psi_s}^*$ is  cyclic on $A^2_{\beta}.$ According to the work of P. S. Bourdon and J. H. Shapiro (see \cite[Proposition 2.7]{sha}), if $T^*$ has \textit{eigenvalues of infinite multiciplicity} (that is, $Ker(T-\lambda I)$ is finite dimensional for all complex number $\lambda$) then $T$ is not cyclic, based on result we show the following:

\begin{lem}
	The operator $C_{\psi_s}$ is not cyclic on $A_{\beta}^2$ for each  $0<\left|s\right|<1.$
\end{lem}
\begin{proof}
By denoting $H_0=\{f\in A^2_{\beta}:f(0)=0\},$ P. R. Hurst in \cite[Theorem  5]{Hurst} showed that $C_{\psi_s}^*|_{H_0}$ is similar to the $sC_{\sigma}$ where $\sigma(z)=sz+1-s.$ In this work he also showed that for each $\mathrm{Re}(\lambda)>-(\beta+2)/2,$ the function $f_{\lambda} \in A^2_{\beta},$ $f_{\lambda}(z)=(1-z)^{\lambda},$ is an eigenvector for $C_{\sigma}$ corresponding to the eigenvalue $\lambda.$ So, for each integer $k$ and $\lambda(k)=\lambda+2\pi i k/\log s,$ we see $C_{\sigma}f_{\lambda(k)}=s^{\lambda}f_{\lambda(k)}.$ In particular, this last equality show that $s^{\lambda}$ is a eigenvalue of infinite multiplicity for $C_{\sigma}$ since the setting $\{f_{\lambda(k)}:k\in \mathbb{Z}\}$ is linearly independent, and thefore $C_{\sigma}^*$ is not cyclic. Since cyclicity is invariant under similarity and $H_0$ reduces the operator $C_{\psi_s},$ it follows that $C_{\psi_s}$ is not cyclic too.
\end{proof}
It follows that hyperbolic linear fractional non-automorphism do not induce complex symmetric composition operators. This is consequence of the strong relation between cyclicity and complex symmetry.

\begin{thm}\label{12} Let $\phi$ a hyperbolic linear fractional map then $C_{\phi}$ is not complex symmetric on $A^2_{\beta}$. 
\end{thm}
\begin{proof} First we consider the non-automorphism case. By Lemma \ref{11}, $C_{\phi}$ is similar to the $C_{\psi_s}$ for some $0<\left|s\right|<1.$  Since cyclicity is invariant under similarity it follows that $C_{\phi}^*$ and $C_{\phi}$ are not simultaneously cyclic, and therefore by Proposition \ref{pps3} we conclude that $C_{\phi}$ is not complex symmetric. The automorphism case follows by Theorem \ref{thm2}.
\end{proof}

Due to Theorem \ref{12} it remains to classify the complex symmetric composition operators induced by linear fractional self-maps with a fixed point inside $\mathbb{D}$ and outside $\overline{\mathbb{D}}.$ It is worth mentioning that even in the Hardy space $H^2(\mathbb{D})$ this case remains open.

\section{The symbol $\phi=\phi_{\alpha}\circ(\lambda\phi_{\alpha})$} \label{5}
In this section the map $\phi$ denotes the linear fractional self-map $\phi_{\alpha}\circ (\lambda \phi_{\alpha})$ where $ \lambda \in \overline{\mathbb{D}}$ and $\phi_\alpha$ are the involutive automorphism (see \eqref{0}). As we saw in \eqref{pre2} the map $\phi$ is an automorphism of $\mathbb{D}$ whenever $\lambda\in \mathbb{T}$. Here, we will see that a location of the numbers $\alpha$ and $\lambda$ play an important in deciding if $C_{\phi}$ is complex symmetric. The first result on the map $\phi$ appeared in \cite[Proposition 2.1.]{Garc1} In this work S. R. Garcia and C. Hammond showed that $\phi_{\alpha}$ and constants always induce complex symmetric composition operators on $A^2_{\beta}.$ In this direction,  we must also highlight the work \cite{wal} where P. S. Bourdon and S. W. Noor presented an almost complete characterization of automorphisms of $\mathbb{D}$ which induce complex symmetric composition operators on $H^2(\mathbb{D}).$ The composition operators induced by automorphisms were studied on $A^2,$ by T. Eklund, M. Lindström, P. Mleczko where they showed that the techniques used on $H^2(\mathbb{D})$ can be adapted for $A^2.$

By comparing the results on $A_{-1}^2:=H^2(\mathbb{D})$ and $A^2:=A^2_0(\mathbb{D}),$ we note that the index $\beta$ strongly influence the complex symmetry of the invertible bounded composition operators. Here we will unify these result by proving a general version of \cite[Proposition 3.3]{wal} and \cite[Theorem 10]{Ted}. We start by studying  $\phi=\phi_{\alpha}\circ(\lambda\phi_{\alpha})$ with $\lambda\in \mathbb{D}.$ If $\lambda=0$ then the symbol $\phi$ is constant while $\alpha=0$ gives $\phi(z)=\lambda z.$ So, we assume $\alpha$ and $\lambda$ are non-zero.
\begin{lem}\label{lem2}
	Let $\beta\in \mathbb{N}.$ If $\alpha\in \mathbb{D}\backslash\{0\}$ and $v_n:=C_{\phi_{\alpha}}^*z^n$ for each non-negative integer $n$ then $v_n$ is orthogonal to $v_0$ whenever $n>2+\beta.$
	\end{lem}
\begin{proof} A simple computation shows that the action of $C_{\phi_{\alpha}}$ on $K_{\alpha}$ is given by
\begin{equation}\label{q51}
C_{\phi_{\alpha}}K_{\alpha}(z)= \frac{(1-\overline{\alpha}z)^{2+\beta}}{(1-\left|\alpha\right|^2)^{2+\beta}},
\end{equation}
and therefore the computation \eqref{sym2} provides
\begin{align}\label{sym3}
\left\langle v_n,v_0\right\rangle=&
\left\langle C_{\phi_{\alpha}}^{*}z^n,K_{\alpha}\right\rangle\nonumber
= \left\langle z^n,C_{\phi_{\alpha}}K_{\alpha}\right\rangle\nonumber
=\displaystyle \frac{1}{(1-\left|\alpha\right|^2)^{2+\beta}}\left\langle z^n,(1-\overline{\alpha}z)^{2+\beta}\right\rangle\nonumber\\
=&\frac{1}{(1-\left|\alpha\right|^2)^{2+\beta}}\sum_{k=0}^{2+\beta} {{2+\beta}\choose k }\left\langle z^n,(-\overline{\alpha}z)^k\right\rangle
.
\end{align}
From \eqref{sym3} we see that $v_n$ is orthogonal to $v_0$ whenever $2+\beta>0.$
\end{proof}
Due to the generalized Newton binomial formula, we see that the conclusion of Lemma \ref{lem2} fails if $2+\beta\notin \mathbb{N}$ since
\begin{align}\label{rem1}
(1-\overline{\alpha}z)^{2+\beta}=\sum_{k=0}^{\infty}{{2+\beta}\choose k }(-\overline{\alpha}z)^k \ \text{and} \ \left\langle v_n,v_0\right\rangle ={2+\beta\choose n}\left\| z^n\right\| ^2.
\end{align}

The Theorem \ref{thm6} shows that in general the converse of Proposition \ref{pps3} is not true More precisely, the cyclicity of $C_{\phi}$ and $C_{\phi}^*$ is not enough to guarantee the complex symmetry of $C_{\phi}.$ In fact, if we consider the symbol $\phi=\phi_{\alpha}\circ(\lambda \phi_{\alpha})$ with $\alpha, \lambda\in \mathbb{D}\backslash \{0\}$ then $\phi$ is neither constant nor elliptic automorphism of $\mathbb{D}$ and $\phi(\alpha)=\alpha.$ These conditions together with \cite[Theorem 3]{Worner} and \cite[Theorem 3.2.]{sha} guarantee that $C_{\phi}^*$ and $C_{\phi}$ are cyclic on $H^2(\mathbb{D}).$ In \cite{anna}, A. Gori studied cyclic phenomena for composition operators on weighted Bergman space, in particular she showed that cyclic operators on $H^2(\mathbb{D})$ are cyclic on $A^2_{\beta}$ too (see \cite[Theorem 1.2]{anna}). Therefore, $C_{\phi}$ and $C_{\phi}^*$ are cyclic on $A^2_{\beta}.$ 

\begin{thm}\label{thm6}
	Let $\beta\in\mathbb{N} .$ If $\phi=\phi_{\alpha}\circ(\lambda \phi_{\alpha})$ with $\alpha, \lambda\in \mathbb{D}\backslash \{0\}$ then $C_{\phi}$ is not complex symmetric on $A^2_{\beta}.$
\end{thm}
\begin{proof}
	Since $C_{\phi}$ is cyclic, the eigenvalues of $C_{\phi}^*$ are simple. Putting $v_n:=C_{\phi_{\alpha}}^*z^n$ for each integer non-negative $n,$ we have  $C_{\phi}^*v_n=\overline{\lambda}^nv_n,$ hence $Ker(C_{\phi}^*-\overline{\lambda}^nI)$ is generated by $v_n.$ Suppose that $C_{\phi}$ is complex symmetric, and let $C$ be a conjugation on $A^2_{\beta}$ such that $CC_{\phi}C=C_{\phi}^*.$  As we saw above, $C_{\phi}\phi_{\alpha}^n=\lambda^n\phi_{\alpha}$ for each non-negative integer $n,$ so the relation \eqref{37}
implies that $C_{\phi}^*C\phi_{\alpha}^n=\overline{\lambda}^nC\phi_{\alpha}^n.$	From what we saw earlier this last equality implies that $C\phi_{\alpha}^n=r_nv_n$ for some complex number $r_n.$ Let $n>2+\beta,$ then Lemma \ref{lem2} implies that $v_n$ is orthogonal to $v_0,$ hence
	\begin{align*}
	\alpha^n=[\phi_{\alpha}(0)]^n=\left\langle \phi_{\alpha}^n,K_0\right\rangle =\left\langle CK_0,C\phi_{\alpha}^n\right\rangle =r_0\overline{r_n}\left\langle v_0,v_n\right\rangle =0.
	\end{align*}
	This last equality forces $\alpha=0,$ contradicting the hypothesis $\alpha\neq 0.$
	
\end{proof}

\subsection{Elliptic automorphism } Let $\phi$ be an elliptic automorphism of $\mathbb{D},$ then it has the form $\phi=\phi_{\alpha}\circ(\lambda\phi_{\alpha})$ for some $\alpha\in \mathbb{D}$ and $\lambda\in \mathbb{T}$ (see \eqref{frac2}). Since $\lambda$ is an unitary number, we define the order of $\phi$ through the number $\lambda.$ More precisely, we say that $\phi$ has \textit{finite order} $N,$ if $N$ is the smallest positive integer for which $\lambda^N=1.$ If no such integer exists then $\phi$ is said to have \textit{infinite order.}

Our next goal is to prove that elliptic automorphisms of infinite order which are not rotations do not induce complex symmetric composition operators on $A^2_{\beta}.$

\begin{thm}\label{thm13}
	Suppose that $\phi$ is an elliptic automorphism of infinite order and not a rotation. Then $C_{\phi}$ is not complex symmetric on $A^2_{\beta}.$
\end{thm}
\begin{proof} We get a contradiction by assuming that $C_{\phi}$ is complex symmetric. Let $C$ a conjugation on $A^2_{\beta}$ such that $CC_{\phi}C=C_{\phi}^*.$ Since $(\phi_{\alpha}^n)_{n\in \mathbb{N}}$ is a sequence of eigenvectors for $C_{\phi}$ corresponding to dinstict eigenvalues, we have $\left\langle C\phi_{\alpha}^n, \phi_{\alpha}^m\right\rangle =0$ if $n\neq m.$ Moreover, a simple computation shows that $\left\langle C\phi_{\alpha}^n,\phi_{\alpha}^n\right\rangle \neq 0$ for each $n.$ Putting $b_n:=\left\langle \phi_{\alpha}^n, C\phi_{\alpha}^n\right\rangle $ we obtain 
\begin{align}\label{thm5}	\left\langle C\phi_{\alpha}^n-\left\| z^n\right\| ^{-2}\overline{b_n}\;v_n, \phi_{\alpha}^m\right\rangle =&\left\langle C\phi_{\alpha}^n, \phi_{\alpha}^m\right\rangle - \frac{1}{\left\| z^n\right\| ^{2}}\overline{b_n}\;\left\langle C_{\phi_{\alpha}}^*z^n,\phi_{\alpha}^m\right\rangle \nonumber\\
=& \left\langle C\phi_{\alpha}^n, \phi_{\alpha}^m\right\rangle - \frac{1}{\left\| z^n\right\| ^{2}}\overline{b_n}\;\left\langle z^n,C_{\phi_{\alpha}}C_{\phi_{\alpha}}z^m\right\rangle \nonumber\\
=& \left\langle C\phi_{\alpha}^n, \phi_{\alpha}^m\right\rangle - \frac{1}{\left\| z^n\right\| ^{2}}\left\langle C\phi_{\alpha}^n, \phi_{\alpha}^m\right\rangle\left\| z^n\right\| ^{2}=0.
\end{align}	
By combining \eqref{thm5} and the a density of the span of $(\phi_{\alpha}^n)_{n\in \mathbb{N}}$ on $A^2_{\beta},$ we reach the following relation 
\begin{align*}
C\phi_{\alpha}^n=\frac{1}{\left\| z^n\right\| ^{2}}\overline{b_n}v_n
\end{align*} 	
for each non-negative integer $n.$ Now consider the map $\psi=\phi_{\alpha}\circ (\delta\phi_{\alpha})$ with $\delta\in \mathbb{D}\backslash \{0\},$ and observe that 
\begin{align}\label{thm7}
CC_{\psi}^*C\phi_{\alpha}^n=\frac{b_n\delta^n}{\left\| z^n\right\| ^{2}}Cv_n
=\frac{b_n\delta^n}{\left\| z^n\right\| ^{2}}\frac{\left\| z^n\right\| ^2}{b_n}\phi_{\alpha}^n=\delta^n\phi_{\alpha}^n=C_{\psi}\phi_{\alpha}^n.
\end{align}
for each non-negative integer $n.$ In particular, \eqref{thm7} forces $CC_{\psi}^*C=C_{\psi},$ that is $C_{\psi}$ is complex symmetric, however this contradicts Theorem \ref{thm6}.
	
\end{proof}

To treat the finite order elliptic automorphism case, we will use a lemma, which relates the iterates of the operator $M_z^*$ acting on the vectors $z^n.$
\begin{lem}\label{216} Let $M_{z}:A^2_{\beta}\longrightarrow A_{\beta}^{2}$ be multiplication by $z$, then for each non-negative integer $n$ and $m$ we have 
	$$
	\left(M_{z}^{*}\right)^{m}z^{n}=\left\{\begin{array}{rl}
	0 &\hspace{0.2cm} for \hspace{0.2cm} m>n\hspace{0.2cm} \\
	\\
	c_{m,n}z^{n-m}&\hspace{0.2cm} for\hspace{0.2cm} 0<m\leq n\\
	\\
	z^{n}&\hspace{0.2cm}for \hspace{0.2cm} m=0
	\end{array}\right.$$
	where $c_{m,n}=\displaystyle\frac{\Gamma(2+n-m+\beta)n!}{(n-m)!\Gamma(2+n+\beta)}.$ 
\end{lem}
\begin{proof} We first fixe $m.$ Following, we note that 
\begin{align}\label{v50}
\left\langle (M_z^{*})^mz^n,f\right\rangle= \left\langle M_{z^m}^{*}z^n,f\right\rangle
= \left\langle z^n,M_{z^m}f\right\rangle
=\left\langle z^n, z^mf\right\rangle
\end{align}
for each non-negative integer $n$ and $f\in A^2_{\beta}.$ From \eqref{v50}, we see that the result is immediate if $m=0.$ So, we assume $m>0.$ Additionally, if $m>n$ the equality \eqref{v50} implies that $z^mf$ is orthogonal to $z^n,$ and therefore $(M_z^*)^mz^n=0.$ If $m\leq n,$ a simple computation using the $A_{\beta}^2$ inner product provides
\begin{align}
\left\langle (M_z^{*})^mz^n,f\right\rangle=& \displaystyle \frac{n!\Gamma(2+\beta)}{\Gamma(2+n+\beta)}\widehat{f}(n-m)\nonumber\\
=&\displaystyle \frac{(n-m)!\Gamma(2+\beta)}{\Gamma(2+n-m+\beta)}\frac{\Gamma(2+n-m+\beta)}{(n-m)!\Gamma(2+\beta)}\frac{n!\Gamma(2+\beta)}{\Gamma(2+n+\beta)}\widehat{f}(n-m)\nonumber\\
=& \displaystyle \left\langle \frac{\Gamma(2+n-m+\beta)}{(n-m)!\Gamma(2+\beta)}\frac{n!\Gamma(2+\beta)}{\Gamma(2+n+\beta)}z^{n-m},f\right\rangle\nonumber\\
=&\displaystyle \left\langle \frac{\Gamma(2+n-m+\beta)n!}{(n-m)!\Gamma(2+n+\beta)}z^{n-m},f\right\rangle\nonumber.
\end{align}
Therefore, $\displaystyle (M_z^{*})^mz^n=\frac{\Gamma(2+n-m+\beta)n!}{(n-m)!\Gamma(2+n+\beta)}z^{n-m}.$
\end{proof}

The next lemma generalizes \cite[Lemma 2.2.]{wal} and \cite[Lemma 5.]{Ted}.
\begin{lem}
	Let $\beta\in \mathbb{N}.$ If $\alpha \in \mathbb{D}\backslash\{0\} $ and $v_n:=C_{\phi_{\alpha}}^*z^n$ for each non-negative integer $n$ then $v_n$ is orthogonal to $v_m$ whenever $\left| n-m\right| \geq \beta+3.$
\end{lem}
\begin{proof} Recall that in \eqref{658} we computated an expression for the adjoint of $C_{\phi_{\alpha}}$ 
\begin{align}\label{q61}
C_{\phi_{\alpha}}^*=   \sum_{k=0}^{2+\beta}{2+\beta\choose k} (-\alpha)^k M_{K_{\alpha}}C_{\phi_{\alpha}}(M_z^*)^k.
\end{align}
We use \eqref{q61} to compute the action of $C_{\phi_{\alpha}}^*$ on $z^n.$ In view of Lemma \ref{216}, we have two cases to consider, namely: $n>2+\beta$ and $ 2+\beta\geq n.$ To simplify the study of these cases, we use the notation
 $r_k={2+\beta\choose k} (-\alpha)^k$ for each $k\in \left\lbrace 0,1,\ldots,2+\beta\right\rbrace. $ If $n>\beta+2$, we obtain
\begin{align}\label{thm8}
C_{\phi_{\alpha}}^*z^n= \displaystyle \sum_{k=0}^{2+\beta}r_kM_{K_{\alpha}}C_{\phi_{\alpha}}\left(c_{k,n}z^{n-k}\right)=\displaystyle \sum_{k=0}^{2+\beta}r_kc_{k,n} K_{\alpha}(z)\left[\phi_{\alpha}(z)\right]^{n-k}.
\end{align}
For the second case, that is $2+\beta\geq n,$ we have
\begin{align}\label{thm9}
C_{\phi_{\alpha}}^*z^n= \displaystyle \sum_{k=0}^{n}r_kM_{K_{\alpha}}C_{\phi_{\alpha}}\left(c_{k,n}z^{n-k}\right)=\displaystyle \sum_{k=0}^{2+\beta}r_kc_{k,n} K_{\alpha}(z)\left[\phi_{\alpha}(z)\right]^{n-k},
\end{align}
where $c_{k,n}=0$ if $k=n+1, \ldots,2+\beta.$ So, a simple computation from \eqref{thm8} and \eqref{thm9} provides

\begin{align}\label{thm10}
C_{\phi_{\alpha}}C_{\phi_{\alpha}}^*z^n=\displaystyle \sum_{k=0}^{2+\beta}r_kc_{k,n} C_{\phi_{\alpha}}K_{\alpha}\phi_{\alpha}^{n-k}(z)=\displaystyle \sum_{k=0}^{2+\beta}r_kc_{k,n}(C_{\phi_{\alpha}}K_{\alpha})(z)z^{n-k}
\end{align}
To obtain the power series of $C_{\phi_{\alpha}}C_{\phi_{\alpha}}^*z^n$ we apply \eqref{q51} in \eqref{thm10} and we again use the Newton binomial formula, getting  
\begin{align}
C_{\phi_{\alpha}}C_{\phi_{\alpha}}^*z^n=&\displaystyle \sum_{k=0}^{2+\beta} \frac{r_kc_{k,n}}{(1-\left|\alpha\right|^2)^{2+\beta}}(1-\overline{\alpha}z)^{2+\beta}z^{n-k}\nonumber\\
=& \displaystyle \sum_{k=0}^{2+\beta} \frac{r_kc_{k,n}}{(1-\left|\alpha\right|^2)^{2+\beta}}\left[\sum_{j=0}^{2+\beta}{2+\beta\choose j}(-\overline{\alpha}z)^j\right]z^{n-k}\nonumber\\
=&\displaystyle \sum_{k=0}^{2+\beta}\sum_{j=0}^{2+\beta}\frac{\overline{r_j}r_kc_{k,n}}{(1-\left|\alpha\right|^2)^{2+\beta}}z^{n-k+j}.\nonumber
\end{align}
By using this last equality we compute the inner product between $v_n$ and $v_m,$ as follows
\begin{equation}\label{thm11}
\left\langle v_n,v_m\right\rangle=\left\langle z^n,C_{\phi_{\alpha}}C_{\phi_{\alpha}}^*z^m\right\rangle =    \sum_{k=0}^{2+\beta}\sum_{j=0}^{2+\beta}\frac{r_j\overline{r_k}\;\overline{c_{k,m}}}{(1-\left|\alpha\right|^2)^{2+\beta}}    \left\langle z^n,z^{m-k+j}\right\rangle.
\end{equation}
From \eqref{thm11} we see that the orthoganality between $v_n$ and $v_m$ is closely linked to indices $j,k,m$ and $n,$ more precisely 
\begin{center}$v_n\perp v_m \Longleftrightarrow \left| m-k+j- n\right| \neq 0.$
\end{center} Since $0\leq j,k\leq 2+\beta,$ the inequality $-\left|j-k\right|\geq -(2+\beta)$ holds. Moreover, by assuming $\left|n-m\right|\geq 3+\beta$ we get
\begin{equation*}\label{218}
\left|m-k+j-n\right|\geq \left|n-m\right|-\left|j-k\right|\geq 3+\beta-(2+\beta)=1.
\end{equation*}
Therefore, $v_n\perp v_m.$ 
\end{proof}
Although the statement above  is an analogue of \cite[Lemma 2.2]{wal}, our proof is very different.

\begin{lem}\label{234} Let $\beta \in \mathbb{N} $ and $\alpha\in \mathbb{D}\backslash\{0\}.$ Suppose also that $\phi=\phi_{\alpha}\circ (\lambda \phi_{\alpha})$ is an elliptic automorphism of finite order $N$ and define $V_n=Ker(C_{\phi}^{*}-\overline{\lambda}^nI)$ for $n=0,1,\ldots,N-1.$ If $N\geq 2(3+\beta)$ then $V_0\perp V_{3+\beta}.$ 
\end{lem}
\begin{proof} For each non-negative integer $n$ we put $v_n:=C_{\phi_{\alpha}}^*z^n.$ We show that $V_n=\overline{span}(v_{kN+n})_{k\in \mathbb{N}}.$ Since $C_{\phi}^*v_{kN+n}=\overline{\lambda}^{n}v_{kN+n}$ for each non-negative integer $k,$ the inclusion $\overline{span}(v_{kN+n})_{k\in \mathbb{N}}\subset V_{n}$ is immediate. On the other hand, $C_{\phi_{\alpha}}^*f\in Ker(C_{\overline{\lambda}z}-\overline{\lambda}^nI)$ for each $f\in V_n,$ since $C_{\phi_{\alpha}}^*(C_{\overline{\lambda} z}-\overline{\lambda}^nI)C_{\phi_{\alpha}}^*=C_{\phi}^*-\overline{\lambda}^nI.$ Since
	\begin{align*}
	Ker(C_{\overline{\lambda}z}-\overline{\lambda}^nI)=\overline{span}(z^{kN+n})_{k\in \mathbb{N}}
	\end{align*}
	for $n=0,1,\ldots, N-1$ we have $C_{\phi_{\alpha}}^*f\in \overline{span}(z^{kN+n})_{k\in \mathbb{N}}$ or equivalently $f\in \overline{span}(v_{kN+n})_{k\in \mathbb{N}}.$ This forces $V_n=\overline{span}(v_{kN+n})_{n\in \mathbb{N}}.$ Then given $v_{kN}\in V_{0}$ and $v_{jN+(3+\beta)}\in V_{3+\beta}$ for each non-negative integer $k$ and $j$ we have
\begin{center}
	$\left|kN-\left[jN+(3+\beta)\right]\right|\geq \left|N(k-j)\right|-(3+\beta)\geq 2(3+\beta)-(3+\beta)=3+\beta.$
\end{center}
By Lemma \ref{216} we conclude that $v_{kN}\perp v_{jN+(3+\beta)},$ or more precisely $V_0\perp V_{3+\beta}.$
\end{proof}
The next result is analogous to \cite[Proposition 3.3]{wal} and it generalizes to all $A_{\beta}^2$ for $\beta\in \mathbb{N}.$

\begin{thm}\label{thm12}Let $\beta\in \mathbb{N} $ and $\alpha\in \mathbb{D}\backslash\{0\}.$ Suppose also that $\phi$ is an elliptic automorphism of finite order $N\geq 2(3+\beta)$ and not a rotation then, $C_{\phi}$ is not complex symmetric on $A_{\beta}^2.$
	
\end{thm}
\begin{proof} First we observe that Lemma \ref{234} guarantees that $V_0\perp V_{3+\beta},$ because $N\geq 2(3+\beta).$ We get a contradiction by assuming that  $C_{\phi}$ is complex symmetric. Let $C$ be a conjugation on $A^2_{\beta}$ such that $C_{\phi}C=CC_{\phi}^*.$ Then $C$ maps $V_0$ and $V_{3+\beta}$ onto $Ker(C_{\phi}-I)$ and $Ker(C_{\phi}-\lambda^{3+\beta}I)$ respectively. Since $C$ preserves orthogonality, we have
\begin{equation}\label{235}
Ker(C_{\phi}-I)\perp Ker(C_{\phi}-\lambda^{3+\beta}I)
\end{equation}
and in particular  \eqref{235} implies that $K_0$ is orthogonal to $\phi_{\alpha}^{3+\beta},$ hence
\begin{align*}{\alpha}^{3+\beta}=\left[\phi_{\alpha}(0)\right]^{3+\beta}=\left\langle \phi_{\alpha}^{3+\beta},K_0\right\rangle=0
\end{align*}
This last equality forces $\alpha=0,$ and contradicts the hypothesis $\alpha\neq 0.$
\end{proof}

Due to Theorem \ref{thm12} it follows that the order $3,4,\ldots, 5+2\beta$ elliptic cases remain open, more precisely:
\\
\\
\textbf{Problem :} Let $\beta\in \mathbb{N}$ and $\phi$ an elliptic automorphism of $\mathbb{D}.$ Is $C_{\phi}$  complex symmetric on $A^2_{\beta}$ when $\phi$ has order $N=2,3,\ldots , 5+2\beta  ?$ 
\\
\\
\textbf{Acknowledgments}
\\

The author is grateful to Professor Sahibzada Waleed Noor for the discussions and suggestions. This work is part of the doctoral thesis of the author.

\bibliographystyle{amsplain}

\end{document}